\documentclass[11pt,reqno]{amsart}
\usepackage{amsmath,amsthm,amsfonts,amssymb,mathrsfs,stmaryrd,color}
\usepackage[all]{xy}
\usepackage{url}
\usepackage{geometry}
\usepackage{amsmath}
\usepackage{amsthm}
\usepackage{amsfonts}
\usepackage{amssymb}
\usepackage{tikz-cd}
\usepackage[numbers]{natbib}
\usepackage{mathrsfs}
\usepackage{enumerate}
\usepackage[utf8]{inputenc}
\usepackage[T1]{fontenc}
\usepackage{hyperref}
\usepackage[capitalize]{cleveref}
\usepackage{enumitem}
\setlist[enumerate]{itemsep=2pt,parsep=2pt,before={\parskip=2pt}}
\let\amsamp=&
\makeatletter
\newcommand{\colim@}[2]{%
  \vtop{\m@th\ialign{##\cr
    \hfil$#1\operator@font colim$\hfil\cr
    \noalign{\nointerlineskip\kern1.5\ex@}#2\cr
    \noalign{\nointerlineskip\kern-\ex@}\cr}}%
}
\newtheorem{theorem}{Theorem}[section]
\newtheorem*{theorem*}{Theorem}
\newtheorem*{definition*}{Definition}
\newtheorem{proposition}[theorem]{Proposition}
\newtheorem{lemma}[theorem]{Lemma}

\theoremstyle{definition}
\newtheorem{definition}[theorem]{Definition}
\newtheorem{question}[theorem]{Question}
\newtheorem{remark}[theorem]{Remark}

\newcommand{\otimesl}{\otimes^{\boldsymbol{\mathrm{L}}}}
\newcommand{\bb}[1]{\mathbb{#1}}

\newcommand{\ideal}[1]{\mathfrak{#1}}
\renewcommand{\to}{\rightarrow}

\newcommand{\cal}[1]{\mathcal{#1}}

\let\lim\relax
\DeclareMathOperator*{\lim}{lim}

\DeclareMathOperator{\Tor}{Tor}

\DeclareMathOperator{\spec}{Spec}

\author{Ivan Zelich}
\address{Department of Mathematics\\Columbia University\\
New York, NY 10027}
\email{ivan.zelich@columbia.edu}
\begin{document}
\title{Affineness of maximal \'{E}tale locus}
\begin{abstract}
In this paper we will prove a strong version of the celebrated purity of the ramification locus theorem in algebraic geometry. 
Our key input is a Tor-independence result for global sections of \'{e}tale schemes over excellent regular local rings, 
which we will prove by tilting to perfect rings.
\end{abstract}
\maketitle
\section{Introduction}
The Zariski-Nagata purity theorem is a fundamental result in algebraic geometry. The following version is proved in \cite[0EA4]{sp} where it is
attributed to Ofer Gabber.
\begin{theorem}[Purity of the ramification locus]\label{thm:nagata-purity}
Let $X \to Y$ be a morphism of locally of finite-type between locally Noetherian schemes. Let $x \in X$ be a point and let $y = f(x)$. 
Assume that $\cal{O}_{X,x}$ is normal, $\cal{O}_{Y,y}$ is regular, and $\dim \cal{O}_{X,x} = \dim \cal{O}_{Y,y}\ge 1$. 
If for every point specialization $x' \leadsto x$ such that $\dim \cal{O}_{X,x'} = 1$, the morphism $f$ is unramified at $x'$, 
then $f$ is \'{e}tale at $x$. \cite[0EA4]{sp}
\end{theorem}
A stronger form of this purity statement has been proved in the equicharacteristic case. The statement is as follows:
\begin{theorem}[Affineness of the complement of the ramification locus, equicharacteristic case]\label{thm:affine-equi}
Let $X \to Y$ be a morphism of finite-type between locally Noetherian schemes, where $Y$ is excellent and regular containing a field and 
$X$ is normal. Let $V \subset X$ be the maximal open subset such that $f|_V: V \to Y$ is \'{e}tale. 
Then the inclusion $V \to X$ is an affine morphism. \cite[0ECD]{sp}.
\end{theorem}
It is an open question whether the same statement holds in mixed characteristic, and we will prove that this is indeed the case. 
Our key input is the following Tor-independence result for global sections of \'{e}tale schemes over excellent regular local rings (see Theorem~\ref{Theorem:torindepmixed} (iii)).
\begin{theorem}\label{thm:intro1}
Let $(A,\ideal{m}_A,k)$ be an excellent regular local ring with $p \in \ideal{m}_A$, and let $V \to \spec{A}$ be an \'{e}tale, quasi-compact and separated morphism. Suppose that
$H^i(\cal{O}_V)$ is an $A$-module supported on the maximal ideal $\{\ideal{m}_A\}$ for $i \ge 1$. Then $H^0(\cal{O}_V) \otimesl_A k$ is a discrete ring.
\end{theorem}
This theorem rests on the recent result of Bhargav Bhatt where he proves the Cohen-Macauleyness of absolute integral closures \cite{bhattcohen}. 
As a corollary, we obtain (see Theorem~\ref{thm:cohpuregen}):
\begin{theorem}\label{thm:intro2}
Let $(A,\ideal{m}_A,k)$ be an excellent regular local ring of dimension $d$ and $V \to \spec{A}$ an \'{e}tale morphism, quasi-compact and separated morphism. Suppose that
$H^i(\cal{O}_V)$ is an $A$-module supported on the maximal ideal $\{\ideal{m}_A\}$ for $i \ge 1$, and $H^{d-1}(\cal{O}_V)=H^{d}(\cal{O}_V)=0$. 
Then $V$ is an affine scheme.
\end{theorem}
Using standard arguments as in \cite[0ECD]{sp}, we can then finally prove our stronger version of the purity of the ramification locus theorem (see Theorem~\ref{thm:affine-mixed}).
\begin{theorem}[Affineness of the complement of the ramification locus]\label{thm:intro3}
Let $X \to Y$ be a morphism of finite-type between locally Noetherian schemes, where $Y$ is excellent and regular and 
$X$ is normal. Let $V \subset X$ be the maximal open subset such that $f|_V: V \to Y$ is \'{e}tale. 
Then the inclusion $V \to X$ is an affine morphism.
\end{theorem}
In \cite[21.12.14]{EGAIV}, the authors conjectured Theorem~\ref{thm:nagata-purity}, and further conjectured Theorem~\ref{thm:intro3}, because it would imply Theorem~\ref{thm:nagata-purity} 
by combining it with their earlier result \cite[21.12.8]{EGAIV}. 
Thus Theorem~\ref{thm:intro3} answers their conjecture in the affirmative.

\subsection*{Acknowledgements}I would like to thank my advisor Johan De Jong for his guidance and support throughout my PhD. I would also like to thank Hanlin Cai for
many helpful discussions and who very early on suggested that Bhatt's result on the Cohen-Macauleyness of absolute integral closures could be useful. 
Thank you also to Bhargav Bhatt for his helpful comments on an earlier draft of this paper.
\section{Preliminaries}
We will use this section as an opportunity to set-up some notation and give a new proof of our main result in characteristic $p$. We will
use cohomological indexing when working with objects in the derived category. 
For an integral scheme $X$ (resp. a domain $R$), we let $K(X)$ (resp. $K(R)$) denote its field of fractions, and for a natural number $n \in \bb{N}$ greater than $1$, we denote $[n]=\{1,2,\ldots,n\}$.
\begin{definition}\label{def:cohpure}
Let $(A,\ideal{m}_A,k)$ be a Noetherian local ring of dimension $d$. We say a separated $A$-scheme $V \to \spec{A}$ is \textit{cohomologically pure in codimension $1$} 
if $H^i(\cal{O}_V)$ is an $A$-module supported on $\{\ideal{m}_A\}$ for $i \ge 1$, 
$H^{d-1}(\cal{O}_V) = H^{d}(\cal{O}_V) = 0$ where $d = \dim{A}$, and $V \to \spec{A}$ is dominant.
\end{definition}
To justify this choice of terminology, we note the following result.
\begin{lemma}\label{lem:cohpuretop}
Let $(A,\ideal{m}_A,k)$ be a Noetherian local ring of dimension $d$, $V \to \spec{A}$ a morphism cohomologically pure in codimension $1$.
\begin{enumerate}
\item[(i)] If there is an open immersion $V \to \spec{A'}$ of $A$-schemes, where $A'$ is a Noetherian regular domain that is integral over $A$, 
then $\spec{A'} - V$ is of pure codimension $1$ in $\spec{A'}$. In particular, $V$ is affine. 
\item[(ii)] Suppose $V \to \spec{A}$ is \'{e}tale, and let 
$g: A \to B$ be the normalization of $A$ in $V$ giving a compactification 
$V \hookrightarrow \spec{B}$ of $V$ by Zariski's main theorem (\cite[02LR]{sp}). If $A$ is excellent and regular with $d \ge 3$, 
and the pair $(A, \ideal{m}_A)$ is henselian,
then $\spec{B} - V$ has dimension at least $2$. In particular, $V \to \spec{A}$ misses the maximal ideal.
\end{enumerate}
\end{lemma}
\begin{proof}
(i): Let $I \subset A'$ define the complement $\spec{A'} - V$. Let $\ideal{p} \subset A'$ be a component of $V(I)$ of height $c$. Then 
\[H_{V(I)}^c(A') \otimes_{A'} A'_{\ideal{p}} \simeq H_{\ideal{p}A'_{\ideal{p}}}^c(A'_{\ideal{p}}) \neq 0.\]
On the other hand, for $2 \le c \le d-1$, we have $H_{V(I)}^c(A') \simeq H^{c-1}(\cal{O}_V)$ is supported on $\{\ideal{m}_{A}\}$, so $H_{V(I)}^c(A') \otimes_{A'} A'_{\ideal{p}} = 0$,
and for $c=d$, we have $H_{V(I)}^d(A') \simeq H^{d-1}(\cal{O}_V) = 0$, so $H_{V(I)}^c(A') \otimes_{A'} A'_{\ideal{p}} = 0$ as well.
Therefore we conclude $c=1$, as desired.\\
(ii): Since $V \to \spec{A}$ is \'{e}tale and $A$ is regular, $V \simeq \prod_{i \in [n]} V_i, n \in \bb{N}$ where each $V_i$ is an integral scheme \'{e}tale over $A$. 
Hence we may assume $V$ is integral.
Then $B$ is a normal domain, and since $A$ is excellent, $B$ is module-finite over $A$. Since
$(A,\ideal{m}_A)$ is henselian, $B$ is a local ring.\\
\indent There are two cases, either $B$ is \'{e}tale over $A$ or it has ramification. In the first case, 
we may conclude by (i),
and in the second case, $V$ is contained in the maximal open subset of $\spec{B}$ where $\spec{B} \to \spec{A}$ is \'{e}tale, 
so the complement $\spec{B} - V$ contains the ramification locus of $g$, which is of pure codimension $1$ in $\spec{B}$ by purity of the ramification locus, 
so $\spec{B} - V$ has dimension at least $2$.
\end{proof}
Let us now illustrate how to prove the characteristic $p$ version of our main result in a different way to the one given in \cite[0ECD]{sp},
 whose ideas will in fact translate to mixed characteristic. We recall the following Tor-independence result for perfect rings.
 \begin{lemma}\label{lem:perftorind}
Let $B \leftarrow A \to C$ be a diagram of perfect rings. Then the natural map $B \otimesl_A C \to B \otimes_A C$ is an isomorphism.
\end{lemma}
\begin{proof}
This has been observed many times. The earliest published reference we are aware of is \cite[Lemma 3.16 or Proposition 11.6]{BhattScholzeProj}.
\end{proof}
\begin{theorem}\label{thm:cohpuregenp}
Let $(A,\ideal{m}_A,k)$ be an excellent regular local ring containing a field of characteristic $p>0$ and $V \to \spec{A}$ a quasi-compact, \'{e}tale morphism.
Then $H^0(\cal{O}_V) \otimesl_A k$ is a discrete ring.
\end{theorem}
\begin{proof}
Let $A_{\text{perf}}$ be the perfection of $A$, and $A'=H^0(\cal{O}_V)$. Since $A$ is regular, $A \to A_{\text{perf}}$ is a faithfully flat map of local rings.
Since the map $k \to k_{\text{perf}}$ is faithfully flat, it suffices to show that $A' \otimesl_A k_{\text{perf}}$ is a discrete ring.
The map $A \to A_{\text{perf}}$ is faithfully flat, so 
\[A' \otimesl_A k_{\text{perf}} \simeq A' \otimesl_A A_{\text{perf}} \otimesl_{A_{\text{perf}}} k_{\text{perf}} \simeq A' \otimes_A A_{\text{perf}} \otimesl_{A_{\text{perf}}} k_{\text{perf}}.\]
Now by flat base-change, $A' \otimes_A A_{\text{perf}} \simeq H^0(\cal{O}_{V'})$ where $V' = V \times_{\spec{A}} \spec{A_{\text{perf}}}$.
Since $V' \to \spec{A_{\text{perf}}}$ is still \'{e}tale, the ring $H^0(\cal{O}_{V'}) = A' \otimes_A A_{\text{perf}}$ is a perfect ring, and hence by Lemma~\ref{lem:perftorind}, we have
\[A' \otimes_A A_{\text{perf}} \otimesl_{A_{\text{perf}}} k_{\text{perf}} \simeq A' \otimes_A A_{\text{perf}} \otimes_{A_{\text{perf}}} k_{\text{perf}} .\]
Hence $A' \otimesl_A k = H^0(\cal{O}_V) \otimesl_A k$ is a discrete ring as claimed.\\
\end{proof}
We then have:
\begin{theorem}\label{thm:cohpuregenequi}
Let $(A,\ideal{m}_A,k)$ be an excellent regular local ring containing a field and $V \to \spec{A}$ an \'{e}tale morphism that is cohomologically pure in codimension $1$. 
Then $V$ is an affine scheme.
\end{theorem}
\begin{proof}
The case $\mathrm{dim}(A) = 2$ is clear and we may therefore assume $d=\mathrm{dim}(A) \ge 3$. 
We may assume the pair $(A, \ideal{m}_A)$ is henselian since the henselianization of $A$ is ind-\'{e}tale over $A$. So by Lemma~\ref{lem:cohpuretop} (ii), we find that $V \to \spec{A}$ misses the maximal ideal of $A$. Hence by base-change, we see $R\Gamma(\cal{O}_V) \otimesl_A k \simeq 0$.
If $\mathrm{char}(k)=0$, we conclude as in \cite[0ECD]{sp}. Our novelty will be when $\mathrm{char}(k)=p>0$.\\
\indent By Theorem~\ref{thm:cohpuregenp}, we have $H^0(\cal{O}_V) \otimesl_A k$ is a discrete ring. Since $V \to \spec{A}$ is cohomologically pure in codimension $1$, 
each $H^i(\cal{O}_V)$ is supported on $\{\ideal{m}_A\}$ and thus has depth $0$ for $i\ge 1$. 
Hence if $H^i(\cal{O}_V) \neq 0$ for $i\ge 1$, then $ H^i(\cal{O}_V) \otimesl_A k$ is non-zero in cohomological degree $-d$. Let $j\in [d-2]$ be the first index such that $H^j(\cal{O}_V)\neq 0$. 
By investigating the spectral sequence
\[E_2^{i,j} = \mathrm{Tor}_{-i}^A(H^j(\cal{O}_V), k) \Rightarrow H^{i+j}(R\Gamma(\cal{O}_V) \otimesl_A k),\]
we conclude that the $E_2^{-d,j}$-term survives to the $E_{\infty}$-page, which contradicts the fact that $R\Gamma(\cal{O}_V) \otimesl_A k \simeq 0$. 
Hence $H^i(\cal{O}_V) = 0$ for $i\ge 1$, and thus $V$ is affine as claimed.
\end{proof}
\section{Mixed characteristic}
Our main result is the following.
\begin{theorem}\label{thm:cohpuregen}
Let $(A,\ideal{m}_A,k)$ be an excellent regular local ring and $V \to \spec{A}$ an \'{e}tale morphism that is cohomologically pure in codimension $1$. Then $V$ is an affine scheme.
\end{theorem}
In the equicharacteristic proof given in the previous section (Theorem~\ref{thm:cohpuregenequi}), we used two critical properties of perfect rings:
\begin{enumerate}
\item For any diagram of perfect rings $B \leftarrow A \to C$, the natural map $B \otimesl_A C \to B \otimes_A C$ is an isomorphism.
\item There exists a faithfully flat perfect cover $A \to A_{\text{perf}}$ such that $H^0(\cal{O}_{V'})=H^0(\cal{O}_V) \otimes_A A_{\text{perf}}$ is a perfect ring, where $V' = V \times_{\spec{A}} \spec{A_{\text{perf}}}$.
\end{enumerate}
Mixed characteristic analogues of the first property are usually stated in the $p$-completed setting, which would not be suitable for our purposes.
Instead, we will use the following.
\begin{proposition}\label{prop:torindep}
Let $B \leftarrow A \to C$ be a diagram of commutative rings, and assume there exists an element $\varpi=pu \in A$ admitting all $p$-power roots where $u \in A$ is a unit. 
Suppose that $C$ is perfect and both $A$ and $B$ are $p$-torsionfree rings such that $A/\varpi^{1/p^{\infty}}$ and $B/\varpi^{1/p^{\infty}}B$ are perfect rings. 
Then $B \otimesl_A C$ is a discrete ring.
\end{proposition}
\begin{proof}
The map $A \to C$ factors over the ring $A/\varpi^{1/p^{\infty}}$.
Note that $B \otimesl_{A} A/\varpi^{1/p^{\infty}} \simeq B/\varpi^{1/p^{\infty}}B$ is a discrete ring since $B$ is $p$-torsionfree (so $\varpi$ is a nonzerodivisor in $B$). Therefore, we have
\[B \otimesl_A C \simeq (B \otimesl_A A/\varpi^{1/p^{\infty}}) \otimesl_{A/\varpi^{1/p^{\infty}}} C \simeq (B/\varpi^{1/p^{\infty}}B) \otimesl_{A/\varpi^{1/p^{\infty}}} C.\]
Since all the rings in the right-term of the above expression are perfect, we conclude that $B \otimesl_A C$ is a discrete ring as desired.
\end{proof}
As for the second property, in mixed characteristic, perfectoid rings are generally considered the correct replacement for perfect rings, 
but finding perfectoid covers is more subtle. Moreover, we
are not sure if the following analogue of the second property holds for perfectoid rings.
\begin{question}
If $V \to \spec{A}$ is an \'{e}tale morphism of schemes with $A$ perfectoid, is the $p$-completion of $H^0(\cal{O}_V)$ a perfectoid ring?
\end{question}
We introduce particular kinds of schemes where some control over the global sections of arbitrary open subsets is guaranteed.
\begin{proposition}\label{prop:perfglobalsections}
Consider the following class of commutative rings having the following property (*): $A$ is a normal domain such that every element of $A$ has a $p$-th root in $A$.
\begin{enumerate}
\item[(i)] If $A$ has property (*), then $A/(p^{1/p^{\infty}})$ is a perfect ring.
\item[(ii)] If $X$ is an integral scheme such that, for every affine open $\spec{A} \subset X$, $A$ has property (*), then $\Gamma(X, \cal{O}_X)$ is also a ring with property (*).
\item[(iii)] If $X$ is an integral scheme such that, for every affine open $\spec{A} \subset X$, $A$ has property (*), then $\Gamma(U, \cal{O}_U)$ has property (*) for every open subset $U \subset X$.
\end{enumerate}
\end{proposition}
\begin{proof}
(i): If $p$ is invertible then $A/(p^{1/p^{\infty}})=0$ and the result is trivial. Otherwise, it is sufficient to show that the Frobenius map $A/p^{1/p^{n+1}} \to A/p^{1/p^{n}}$ is an isomorphism 
for any $n\ge 1$. 
Surjectivity follows because $A$ has all $p$-th roots. For injectivity,
if $x \in A$ has the property that $x^p = p^{1/p^{n}}y$ for some $y \in A$, then $z=\frac{x}{p^{1/p^{n+1}}} \in K(A)$ satisfies the $A$-integral equation $z^p-y=0$.
Since $A$ is normal, we conclude that $z \in A$, so $x = p^{1/p^{n+1}}z \in p^{1/p^{n+1}}A$ as desired.\\
\indent For (ii), by \cite[0358]{sp}, we know that $\Gamma(X, \cal{O}_X)$ is a normal domain. Thus, we need to show that every element $x \in \Gamma(X, \cal{O}_X)$ has a $p$-th root in $\Gamma(X, \cal{O}_X)$. 
Let $\eta \in X$ be the generic point. Then $x|_{\eta} \in k(\eta)$ has a $p$-th root in $k(\eta)$, say $y$. Now $y$ is integral over $\Gamma(X, \cal{O}_X)$, so therefore is also integral 
over any affine open $U=\spec{A} \subset X$.
Since $A$ is normal with fraction field $k(\eta)$, we conclude that there exists an element $y_U \in A$ such that $y_U|_{\eta} = y$ and $y_U^p = x|_U$. It is clear
that for any smaller affine open $V \subset U$ that $y_V = y_U|_V$, hence, we obtain a global section $y \in \Gamma(X, \cal{O}_X)$ such that $(y^p - x)|_U = 0$ for all 
affine opens $U \subset X$, so $y^p - x=0$ in $\Gamma(X, \cal{O}_X)$ as desired.\\
\indent (iii) follows from (ii) by replacing $X$ with $U$.
\end{proof}
There is an obvious choice of schemes satisfying the condition in Proposition~\ref{prop:perfglobalsections}, namely, the spectra of absolute integral closures of local domains.
\begin{proposition}\label{prop:flatcovermixed}
Let $X$ be an integral, normal scheme whose fraction field is algebraically closed.
\begin{enumerate}
\item[(i)] For every affine open $\spec{A} \subset X$, $A$ has property (*) of Proposition~\ref{prop:perfglobalsections}. 
\item[(ii)] For every open subset $U \subset X$, the ring $A'=H^0(\cal{O}_U)$ has property (*) of Proposition~\ref{prop:perfglobalsections}. In particular,
$A'/p^{1/p^{\infty}}$ is a perfect ring.
\item[(iii)] Let $U \to X$ be any separated \'{e}tale morphism. Then $U$ is a disjoint union of open subsets of $X$.
\end{enumerate}
\end{proposition}
\begin{proof}
(i): Every such $A$ is normal with algebraically closed fraction field. Therefore, it suffices to observe that $K(A)$ contains all $p$-th roots.\\
(ii): This follows from (i) and Proposition~\ref{prop:perfglobalsections} (iii).\\
(iii): This is the content of (1) of \cite[09Z9]{sp}.
\end{proof}
To proceed, we will use a critical result of Bhatt in his seminal paper \cite{bhattcohen}.
\begin{theorem}\label{thm:bhattcohen}
Let $(A,\ideal{m}_A,k)$ be an excellent local Noetherian domain with $p \in \ideal{m}_A$. Let $A^{+}$ be the normalization of $A$ in an algebraic closure of its fraction field. 
Then $R\Gamma_{\ideal{m}_A}(A^{+})$ is concentred in cohomological degree $d$ where $d=\mathrm{dim}(A)$.
\end{theorem}
\begin{theorem}\label{Theorem:torindepmixed}
Let $(A,\ideal{m}_A,k)$ be an excellent regular local Noetherian ring of dimension $d$ with $p \in \ideal{m}_A$, $A^{+}$ the normalization of $A$ in the algebraic
closure of $K(A)$, and $i: A \to A^{+}$ the induced integral map.
\begin{enumerate}
\item[(i)] Let $M$ be a module over $A$ that is supported on $\{\ideal{m}_A\}$. Then $M \otimesl_A A^{+}$ is discrete.
\item[(ii)]Let $V \to \spec{A}$ be a flat, quasi-compact morphism of schemes such that $V$ is cohomologically pure in codimension $1$ over $\spec{A}$.
 Then $H^0(\cal{O}_V) \otimesl_A A^{+}$ is a discrete ring equal to $H^0(\cal{O}_{V'})$ where $V'= V \times_{\spec{A}} \spec{A^{+}}$.
 \item[(iii)] In the situation of (ii), if $V \to \spec{A}$ is further assumed to be \'{e}tale, then $H^0(\cal{O}_V) \otimesl_A k$ is a discrete ring.
\end{enumerate}
\end{theorem}
\begin{proof}
(i): Since $M$ is supported on $\{\ideal{m}_A\}$, we have
\[A^{+} \otimesl_A M \simeq R\Gamma_{\ideal{m}_A}(A^{+}) \otimesl_A M.\]
Since $A$ is regular, $A$ has global dimension $d$, and since $R\Gamma_{\ideal{m}_A}(A^{+})$ is concentrated in cohomological degree $d$ by Theorem~\ref{thm:bhattcohen}, 
the right-hand side is concentrated in cohomological degrees $\ge 0$. But the left-hand side is concentrated in cohomological degrees $\le 0$,
so we conclude that $A^{+} \otimesl_A M$ is discrete.\\
(ii): By flat base-change, we know that
\[R\Gamma(\cal{O}_{V'}) \simeq R\Gamma(\cal{O}_V) \otimesl_A A^{+}.\]
We have a spectral sequence 
\[E_2^{i,j} = \mathrm{Tor}_{-i}^A(H^j(\cal{O}_V), A^{+}) \Rightarrow H^{i+j}(R\Gamma(\cal{O}_V) \otimesl_A A^{+}) = H^{i+j}(\cal{O}_{V'}).\]
Since $V \to \spec{A}$ is cohomologically pure in codimension $1$, each $H^j(\cal{O}_V)$ is supported on $\{\ideal{m}_A\}$ for $j \ge 1$, so by (i), $E_2^{-i,j} = 0$ for all $i \ge 1$ and $j \ge 1$.
Therefore, the terms $E_2^{-i,0} = \Tor_{i}^A(H^0(\cal{O}_V), A^{+})$ survive to the $E_{\infty}$-page, contributing to negative cohomological degrees of the complex $R\Gamma(\cal{O}_{V'})$.
We therefore conclude that $H^0(\cal{O}_V) \otimesl_A A^{+}$ is discrete and equal to $H^0(\cal{O}_{V'})$.\\
(iii): By (ii), $H^0(\cal{O}_{V}) \otimesl_A A^{+}$ is a discrete ring equal to $H^0(\cal{O}_{V'})$. Moreover,
by Proposition~\ref{prop:flatcovermixed}, $V'=\prod_{i \in [n]} V_i, n \in \bb{N}$ where each $V_i$ is an open subset of $\spec{A^{+}}$,
so $H^0(\cal{O}_{V_i})$ is a normal domain such that $H^0(\cal{O}_{V_i})/p^{1/p^{\infty}}$ is a perfect ring for each $i \in [n]$.\\
\indent Let $k_{A^{+}}$ be the residue field of $A^{+}$. Note that the induced map 
\[k=A/(\ideal{m}_{A^{+}} \cap A) \to A^{+}/\ideal{m}_{A^{+}}=k_{A^{+}}\]
identifies $k_{A^{+}}$ as an algebraic closure of $k$, so $k \to k_{A^{+}}$ is faithfully flat and $k_{A^{+}}$ is perfect. 
Therefore, by faithfully flat descent, it suffices to show that $H^0(\cal{O}_V) \otimesl_A k_{A^{+}}$ is a discrete ring. We have
\[H^0(\cal{O}_V) \otimesl_A k_{A^{+}} \simeq H^0(\cal{O}_V) \otimesl_A A^{+} \otimesl_{A^{+}} k_{A^{+}} \simeq H^0(\cal{O}_{V'}) \otimesl_{A^{+}} k_{A^{+}}.\]
Since $V'=\prod_{i \in [n]} V_i$, we have
\[H^0(\cal{O}_{V'}) \otimesl_{A^{+}} k_{A^{+}} \simeq \prod_{i \in [n]} H^0(\cal{O}_{V_i}) \otimesl_{A^{+}} k_{A^{+}}.\]
By Proposition~\ref{prop:torindep}, $H^0(\cal{O}_{V_i}) \otimesl_{A^{+}} k_{A^{+}}$ 
is a discrete ring for each $i \in [n]$. Hence $H^0(\cal{O}_V) \otimesl_A k$ is a discrete ring as claimed.\\
\end{proof}
\begin{remark}
As suggested to the author by Bhargav Bhatt, Theorem~\ref{Theorem:torindepmixed} (iii) can be proved more directly if 
one additionally knows that $H^1(\cal{O}_V)=0$. Indeed, since $A$ is excellent, we may assume $A$ is a complete local ring. 
Then $A$ admits a faithfully flat local ring map $A \to A'$ with $A'$ perfectoid (see \cite[Theorem 4.7 or Example 3.8]{padickunz}), 
and $A'$ comes equipped with an element $\varpi=pu \in A'$ admitting all $p$-power roots, where $u \in A'$ is a unit, such that $\bar{A'}=A'/\varpi^{1/p^{\infty}}$ is a perfect ring. Let $V' = V \times_{\spec{A}} \spec{A'}$.
By considering the relevant Bockstein sequences and commuting filtered colimits with cohomology, we see that $H^0(\cal{O}_{V'})/\varpi^{1/p^{\infty}} = H^0(\cal{O}_{\bar{V'}})$, where $\bar{V'} = V' \times_{\spec{A'}} \spec{\bar{A'}}$. Since $\bar{V'} \to \bar{A'}$ is \'{e}tale, $H^0(\bar{V'})$ is a perfect ring. 
Hence, by Proposition~\ref{prop:torindep}, $H^0(\cal{O}_{V'}) \otimesl_{A'} k_{A'}$ is a discrete ring, where $k_{A'}$ is the residue field of $A'$, and since both $A \to A'$ and $k \to k_{A'}$ are faithfully flat, we conclude that $H^0(\cal{O}_V) \otimesl_A k$ is a discrete ring as well.\\
\indent In the context of Theorem~\ref{thm:cohpuregen}, the vanishing of $H^1(\cal{O}_V)$ is equivalent to the affineness of $V$ in the following sense. Assume $\mathrm{dim}(A) \ge 3$. 
For any regular element $f \in A$ one can show that $V_{f=0} \to \spec{A/f}$ 
is still cohomologically pure in codimension $1$, so since the case $\mathrm{dim}(A)=2$ is immediate, by induction on the dimension of $A$, 
we could then conclude that $V_{f=0}$ is affine. 
Hence, if $H^1(\cal{O}_V)=0$, then by examining the corresponding Bockstein sequence, we may conclude that $V$ is affine as well.
\end{remark}
We are now ready to prove Theorem~\ref{thm:cohpuregen}.
\begin{proof}[Proof of Theorem~\ref{thm:cohpuregen}]
We've already resolved the equicharacteristic case in Theorem~\ref{thm:cohpuregenequi}, so we may assume 
that $p \in \ideal{m}_A$. The dimension $2$ case is clear, so we may assume that $d=\mathrm{dim}(A) \ge 3$. 
We may also assume the pair $(A, \ideal{m}_A)$ is henselian since the henselianization of $A$ is ind-\'{e}tale over $A$.\\
\indent By Lemma~\ref{lem:cohpuretop} (ii), 
$V$ misses the closed point of $\spec{A}$, so $R\Gamma(\cal{O}_V) \otimesl_A k \simeq 0$, and by Theorem~\ref{Theorem:torindepmixed} (iii), $H^0(\cal{O}_V) \otimesl_A k$ is a discrete ring. 
Since $V$ is cohomologically pure in codimension $1$, each $H^i(\cal{O}_V)$ is supported on $\{\ideal{m}_A\}$ for $i \ge 1$, 
so if $H^i(\cal{O}_V) \neq 0$ for some $i \ge 1$, then $H^i(\cal{O}_V) \otimesl_A k$ is non-zero in cohomological degree $-d$.
Let $j\in [d-2]$ be the smallest number such that $H^j(\cal{O}_V) \neq 0$. By investigating the spectral sequence
\[E_2^{i,j} = \mathrm{Tor}_{-i}^A(H^{j}(\cal{O}_V), k) \Rightarrow H^{i+j}(R\Gamma(\cal{O}_V) \otimesl_A k),\]
we conclude that the $E_2^{-d,j}$-term survives in the $E_{\infty}$-page, which contradicts the fact that $R\Gamma(\cal{O}_V) \otimesl_A k \simeq 0$.
\end{proof}
\section{Affineness of the maximal \'{e}tale locus}
We establish Theorem~\ref{thm:intro3}.
\begin{theorem}\label{thm:affine-mixed}
Let $X \to Y$ be a morphism of finite-type between locally Noetherian schemes, where $Y$ is excellent and regular and $X$ is normal. 
Let $V \subset X$ be the maximal open subset such that $f|_V: V \to Y$ is \'{e}tale. The inclusion $V \to X$ is an affine morphism.
\end{theorem}
\begin{proof}
The proof follows the same strategy as \cite[0ECD]{sp}, and we include the details for completeness.\\
\indent We first reduce to the situation that $Y$ is has finite dimension.
For any $x \in X$, we need to find an affine open neighbourhood $U$ of $x$ such that $U \cap V$ is affine. 
It therefore necessary and sufficient to show that $V_x = V \times_X \spec{\cal{O}_{X,x}}$ is affine for every $x \in X$ \cite[01Z6]{sp}. 
Hence, it is necessary and sufficient to show the result for the base-change of the morphism
$X \times_Y \spec{\cal{O}_{Y,y}} \to \spec{\cal{O}_{Y,y}}$ for each $y \in Y$, so we may assume $Y$ is local of finite Krull-dimension. 
By induction on the dimension $d$ of $Y$, we may also assume that $V \cap (X \setminus f^{-1}(y)) \to X \setminus f^{-1}(y)$ is affine.\\
\indent For any $x \in f^{-1}(y)$, if $x \in V$ then $V_x$ is clearly affine.
Since $V \to Y$ is quasi-finite, $V \cap f^{-1}(y)$ is a finite disjoint union of closed points, so if $x \in f^{-1}(y)$ is not in $V$, 
then $x$ has an open affine neighbourhood $U$ not meeting $f^{-1}(y)$. By replacing $X$ by $U$ we've reduced to the case that $X$ is affine, $V \to Y$ misses the closed point $y \in Y$, and 
$V \to \spec{\cal{O}_{Y,y}}-\{y\}$ is an affine morphism. If $d =0$ or $1$, then $V$ is affine. Hence, we assume $d \ge 2$.\\
\indent Let $e(V)$ be the largest $i$ such that $H^i(\cal{O}_V) \neq 0$. As $X$ is affine, $X \setminus V$ is defined be an ideal $I \subset \cal{O}_X$, and by \cite[0DXB]{sp}, we have
\[e(V)= \max_{x \in X} e(V_x) - 1, \; \; e(V_x) = \max \{i: H^i_{I\cal{O}_{X,x}}(\cal{O}_{X,x}) \neq 0\}.\]
If $x$ is not a specialization of any point in $V$, then $V_x$ is empty and thus $e(V_x)=0$.
If $x$ is a specialization of some point in $V$, then $\mathrm{dim}(\cal{O}_{X,x}) \le d$ by the dimension formula \cite[02JU]{sp}. $V_x$ is contained in the punctured 
spectrum of $\spec{\cal{O}_{X,x}}$, so by the purity of the ramification locus $X \setminus V$, if
$\mathrm{dim}(\cal{O}_{X,x}) \ge 2$, then 
the dimension of the complement of $V_x$ in $\spec{\cal{O}_{X,x}}$ is at least $1$, and so by Hartshorne-Lichtenbaum vanishing (\cite[0EB7]{sp}), 
$e(V_x) \le \mathrm{dim}(\cal{O}_{X,x})-1 \le d-1$. If $\mathrm{dim}(\cal{O}_{X,x}) \le 1$, then $e(V_x) \le 1$, and therefore, $e(V) \le \max(0,d-2)=d-2$.\\
\indent As $V \to \spec{\cal{O}_{Y,y}} - \{y\}$ is affine and quasi-compact, for $i \ge 1$, $H^i(\cal{O}_V)$ is supported on $\{y\}$. Moreover, $V \to \spec{\cal{O}_{Y,y}}$ 
is separated, and since we know $e(V) \le d-2$, we conclude that $V \to \spec{\cal{O}_{Y,y}}$ is an \'{e}tale morphism that is cohomologically pure in codimension $1$.
By Theorem~\ref{thm:cohpuregen}, $V$ is affine as desired.
\end{proof}
\bibliographystyle{amsalpha}
\bibliography{references}
\end{document}